\documentclass[11pt]{article}

\usepackage[utf8]{inputenc}
\usepackage[T1]{fontenc}
\usepackage{mathpazo}

\RequirePackage[english]{babel}

\RequirePackage[a4paper,margin=2.5cm]{geometry}
\RequirePackage{parskip}

\makeatletter
\def\@fnsymbol#1{\ensuremath{\ifcase#1\or *\or \dagger\or \ddagger\or \mathsection\or \|\or **\or \dagger\dagger \or \ddagger\ddagger \else\@ctrerr\fi}}
\makeatother

\usepackage{xcolor}
\definecolor{cblue}{RGB}{0,70,140}
\definecolor{cgreen}{RGB}{100,140,0}
\definecolor{cred}{RGB}{190,10,50}

\usepackage[shortlabels]{enumitem}
\setlist{itemsep=0ex,topsep=0ex,parsep=0.4ex}

\usepackage{amsmath,amsfonts,amssymb,amsthm,mathtools,thmtools,thm-restate,bbm,centernot}

\usepackage[hyphens]{url} 
\usepackage[hidelinks,colorlinks,pagebackref]{hyperref} 
\hypersetup{citecolor=cgreen,linkcolor=cblue,urlcolor=cblue}
\usepackage[capitalise,nameinlink,noabbrev,compress]{cleveref} 

\renewcommand*{\backref}[1]{}
\renewcommand*{\backrefalt}[4]{
	\ifcase #1 Not cited.%
	\or $\uparrow$#2%
	\else $\uparrow$#2%
	\fi%
}

\let\oldbibliography\bibliography
\renewcommand{\bibliography}[1]{
  {

    \hypersetup{linkcolor=cred}
    \bibliographystyle{bibstyle}
    \oldbibliography{#1}
  }
}

\theoremstyle{plain}
\newtheorem{theorem}{Theorem}[section]
\newtheorem{lemma}[theorem]{Lemma}

\newtheorem{proposition}[theorem]{Proposition}

\theoremstyle{definition}

\makeatletter
\renewenvironment{proof}[1][\proofname]
{\par\pushQED{\qed}
	\normalfont\topsep6\p@\@plus6\p@\relax\trivlist
	\item[\hskip\labelsep\bfseries#1\@addpunct{.}]
	\ignorespaces}
{\popQED\endtrivlist\@endpefalse}
\makeatother

\newcommand{\bad}{\text{$\xi$-primitive}}
\newcommand{\good}{\text{non-$\xi$-primitive}}

\newcommand{\ex}{{\mathrm{ex}}}

\title{Online Ramsey turnaround numbers}

\author{\hfill Nóra Almási\thanks{Budapest University of Technology and Economics, Hungary, {\tt almasi.nora@cs.bme.hu}}
\and Maria Axenovich\thanks{Karlsruhe Institute of Technology, Germany, {\tt maria.aksenovich@kit.edu}} }

\date{\today}

\begin{document}

\maketitle

\begin{abstract}
 The {\it online Ramsey turnaround game} is a game between two players, Builder and Painter, on a board of $n$ vertices using $3$ colors, for a fixed graph $H$ on at most $n$ vertices.
The goal of Painter is to force a monochromatic copy of  $H$, the goal of  Builder is to avoid this as long as possible.  In each round of the game, Builder exposes one new edge and is
allowed to forbid the usage of one color for Painter to color this newly exposed edge, and Painter colors the edge according to this restriction. 
The game is over as soon as Painter manages to achieve a monochromatic copy of $H$. 
For sufficiently large $n$, we consider the smallest number $f(n, H)$ of edges so that Painter can always win after $f(n, H)$ edges have been exposed by Builder.
In addition, we define $f(H)$ to be the smallest $n$ such that Painter can always win on a clique with $n$ vertices.
We give bounds for both functions and show that this problem is closely related to other concepts in extremal graph theory, such as polychromatic colorings, set-coloring Ramsey numbers, chromatic Ramsey numbers, and 2-color Tur\'an numbers.
\end{abstract}

\section{Introduction}
\label{sec:intro}

A classical {\it online Ramsey game} for a graph $H$  is a game between two players, Builder and Painter. 
Starting from an infinite set of isolated vertices, Builder draws an edge on each turn, and Painter immediately paints it red or blue. Builder’s goal is to force Painter to create a monochromatic copy of $H$ using
as few turns as possible. The online Ramsey number for $H$ is the minimum number of edges Builder
needs to guarantee a win in this online Ramsey game. This game was introduced independently by Beck \cite{B} and by Kurek and Ruciński \cite{KR}, and since then, it has attracted a lot of attention, see e.g., papers of Conlon \cite{C2010} and Conlon, Fox, Grinshpun, and He \cite{CFGH}.

Here, we consider a similar game, where the goal of Painter is not to avoid a monochromatic copy of $H$, but rather to quickly force the existence of such a monochromatic copy.
For this game to make sense, one needs to restrict the number of vertices to be $n$, for some natural number $n$. Of course Painter has no incentive to use more than one color, so Builder could only avoid a monochromatic copy of $H$ if the number of edges he draws is at most $\ex(n,H)$, the largest number of edges in an $n$-vertex graph without a copy of $H$.
However, when Builder can forbid certain colors to be used on the exposed edges, the game becomes non-trivial. The creation of such a game is due to Mirbach \cite{M-thesis}, who called it the Ramsey turnaround game.

The {\it online Ramsey turnaround game} for a graph $H=(V,E)$ is a game between two players, Builder and Painter, on a board of $n$ vertices, using $3$ colors: red, green, and blue. The goal of Painter is to force a monochromatic copy of  $H$, the goal of  Builder is to avoid this as long as possible. 
In each round of the game, Builder exposes one new edge and is allowed to forbid the usage of one color for Painter to color this newly exposed edge, Painter colors the edge according to this restriction.  The game is over as soon as Painter manages to achieve a monochromatic copy of $H$.

We treat two different settings. In the first setting, we want to find the smallest number of vertices, $n$, so that Painter has a winning strategy on a board with $n$ vertices. In the second setting, we want to find the smallest number $k=k(n)$  of edges so that  Painter can always win using any $k$  exposed edges on a board with a given, sufficiently large number $n$ of vertices. 
 Next we define the respective functions $f(H)$ and $f(n, H)$. In addition, we shall define the functions $\xi(H)$ and $\xi(n, H)$  giving lower bounds on $f(H)$ and $f(n, H)$, respectively.

Let $f(H)$ be the smallest $n$ such that Painter wins the game for $H$ on $n$ vertices. 
Let $\xi(H)$  
be the largest integer $n$ such that there is an edge coloring of the complete graph $K_n$ in which every subgraph isomorphic to $H$ contains all three colors.
Let $R(H)=R(H;2)$ be the {\it Ramsey number} for $H$, that is, the smallest integer $N$ such that any edge coloring of a complete graph on $N$ vertices in two colors results in a monochromatic copy of $H$.
We have that $f(H)  \leq R(H)$, provided by the following strategy for Painter.
Painter uses only $2$ colors, say, red and blue.  Whenever green is forbidden, Painter may choose blue or red arbitrarily.  After all the edges of a complete graph on $R(H)$ vertices are exposed and colored with two colors, there is a monochromatic copy of $H$, i.e., Painter won the game.
On the other hand, we have that $f(H) \geq \xi(H)+1$, provided by the following strategy for Builder, that we call {\it offline polychromatic strategy}.
Builder considers a complete graph $G$ on $\xi(H)$ vertices with an edge coloring  $c$ in red, blue, and green with every copy of $H$ containing all three colors. Builder exposes the edges $G$ in any order and forbids $c(e)$ for each exposed edge.  After all the edges of $G$ have been exposed, in each copy of $H$ there are three edges that are respectively not red, not blue, and not green. I.e., each copy of $H$ is not monochromatic. Thus we have the following general bounds on $f(H)$:
\begin{equation}\label{xi}
 \xi(H)+1  \leq f(H)\leq R(H).
\end{equation}

Let $f(n,H)$ be the smallest number of rounds (or exposed edges) that Painter needs to force a monochromatic $H$ when playing on $n$ vertices.
Let $\xi(n,H)$, the {\it 3-extremal polychromatic function},  be the largest number of edges in an $n$-vertex graph  $G$ that can be edge colored with three colors such that every copy of $H$ contains all three colors.
Let $\ex_2(n,H)$ be the {\it two-color Tur\'an number} for $H$, that is the largest number of edges in an $n$-vertex graph whose edges could be colored with two colors avoiding a monochromatic copy of $H$.
We have that $f(n,H)  \leq \ex_2(n,H)+1$, provided by the following strategy for Painter.
Painter uses only $2$ colors, say, red and blue.  Whenever green is forbidden, Painter may choose blue or red arbitrarily.  After $\ex_2(n, H)+1$ edges on $n$ vertices are exposed and colored with two colors, there is a monochromatic copy of $H$, i.e., Painter won the game.
On the other hand, we have $f(n, H)\geq \xi(n, H) +1$ using a strategy similar to the above for Builder. This gives the bounds on $f(n,H)$:
\begin{equation}\label{eq:lb}
\xi(n, H)+1  \leq f(n, H)\leq \ex_2(n, H).
\end{equation}

Let $K_t$ and $P_t$ be a clique and a path on $t$ vertices, respectively.
Let $M_t$ be a matching on $t$ edges.

\begin{proposition}[Cliques] \label{cliques}
We have $1.224^t \leq  \xi(K_t)\leq 3.375^t$ and
$\xi(K_t)+1 \leq f(K_t) \leq 3.8^{t+o(t)}.$
In addition,  $\xi(K_3)=4$, $f(K_3)=6$,  $\xi(K_4)=9$, $10\leq f(K_4) \leq 18$, $19 \leq \xi(K_5) \leq 28$, and $20 \leq f(K_5) \leq 46$.
\end{proposition}

\begin{proposition}[Matchings and paths]\label{matchings-paths}
 For any $t\geq 1$, $\lfloor\tfrac{7t-1}{3} \rfloor \leq \xi(M_{t}) \leq f(M_{t})- 1\leq 3t-2$.
 For any $t\geq 2$,   $\lfloor\frac{7}{6}(t-1) \rfloor  \leq \xi(P_t) \leq f(P_t)- 1\leq \tfrac{3}{2}t  -2$.
\end{proposition}

\begin{proposition}[Stars]\label{stars}
   For  $t\geq 3$ we have $\xi(K_{1,t}) = \lfloor\frac{3t-1}{2} \rfloor$ and 
$2t-1 \leq f(K_{1,t}) \leq 2t$.
\end{proposition}

We see that $f(H)$ is finite since $R(H)$ is finite, using (\ref{xi}). Are there graphs $H$ such that $f(H)$ achieves the smallest possible value, that is $f(H)=|V(H)|$?  I.e. can  Painter  win already on $n=|V(H)|$ vertices?  We show that the answer to this question is "No" for any sufficiently large connected graph. 

\begin{lemma} \label{f=n}
There is $n_0\in \mathbb{N}$ such that for any $n\geq n_0$ and any connected graph $H$ on $n$ vertices  $f(H)>n$, i.e., Painter loses on $n$ vertices.
\end{lemma}

We prove this result by giving a
strategy for Builder. However, it is unclear whether Builder could achieve this win using offline polychromatic strategy. We show that for some small graphs $H$, the offline polychromatic strategy fails for Builder even on $|V(H)|$ vertices. Let the {\it broom } graph $B_{s,t}$ be a tree on $s+t$ vertices that is a union of a star with $s$ leaves and a path with $t$ vertices such that the star and the path share exactly one vertex that is the center of the star and an end-point of the path. Let $S_{k,l}$ be the {\it double star}, i.e., the tree on $k+\ell+2$ vertices obtained by taking two vertex-disjoint stars $K_{1,k}$ and $K_{1,\ell}$ and joining their centers with an edge.

\begin{lemma} \label{xi=n}
If $H\in \{P_4, P_5, P_6, B_{3,3}, B_{3,4}, S_{2,2}\}$, then $\xi(H)=|V(H)|-1$.
\end{lemma}

We call graphs $H$, for which $\xi(H)<|V(H)|$, $\xi$-{\it primitive}. These graphs are of independent interest as the ones not having polychromatic property. That is, for any \bad ~graph $H$, any edge coloring of a clique on $|V(H)|$ vertices with three colors has a bi-colored copy of $H$. 
It is unclear whether there are arbitrarily large $\bad$~graphs.

Next, we give bounds on $f(n, H)$ and $\xi(n,H)$. Let $\chi(H)$ denote the chromatic number of a graph $H$. The chromatic Ramsey number, $R_\chi(H)$ of $H$ is the smallest $n$ such that any two-coloring of the edges of $K_n$ contains a homomorphic image of $H$. Equivalently, it is the smallest $k$ such that any two-edge coloring of a sufficiently large complete $k$-partite graph contains a monochromatic copy of $H$.

\begin{proposition}\label{f(n,H)}
For any non-bipartite graph $H$, $$\left( 1 - \frac{1}{\xi(K_{\chi(H)})}\right) \binom{n}{2} (1+o(1))\leq f(n, H) \leq  \left( 1- \frac{1}{R_\chi(H) -1}\right)\binom{n}{2} (1+o(1)).$$
In particular, when $H$ is a clique, we have 
$$\left( 1 - \frac{1}{\xi(K_t)}\right) \binom{n}{2} (1+o(1)) \leq f(n,K_t)\leq \left(1 -\frac{1}{R(K_t)-1} \right)\binom{n}{2} (1+o(1)).$$
For any bipartite $H$, $$ \ex(n, H)\leq f(n, H) \leq  2\ex(n, H)(1+o(1)).$$
\end{proposition}

Coming back to the function $f(H)$ and its lower bound $\xi(H)$, we note that $\xi(H)$ corresponds to several other extremal graph theoretic functions and is related to multiple coloring concepts. Recall that $\xi(H)$ is the largest integer $n$ such that the edges of $K_n$ can be colored in three colors so that each copy of $H$ receives all three colors. This coloring is called
 {\it polychromatic} for $H$ and was studied in particular by Axenovich, Goldwasser, Hansen, Lidick\'y,  Martin,  Offner,  Talbot, and Young \cite{AGHLMOTY}, Goldwasser and Hansen \cite{GH}, and Hansen \cite{H}.
When $H=K_p$, the function $\xi$ can be expressed in terms of the {\it generalized Ramsey number} $f(n, p, 3)$ introduced and studied by Erd\H{o}s and Gy\'arf\'as \cite{EG}. The function $f(n, p, 3)$ is the smallest number of colors used on the edges of $K_n$ such that each copy of $K_p$ has at least three colors. So, if $f(n, p, 3)=3$  and $f(n+1, p, 3)>3$, then $\xi(K_p)=n$. Most results on $f(n, p, q)$ function are concerned with large $n$ and fixed $p$, see \cite{SS}, \cite{CFLS}, \cite{CH} and \cite{BEHK}. 
In addition, $\xi(H)+1$  corresponds to the {\it set-coloring Ramsey number} $R_{3,2}(H)$. 
Here $R_{3,2}(H)$ is the smallest $n$, such that for any assignment of lists of $2$ colors out of red, blue, and green to the edges of $K_n$ there is a copy of $H$ whose edges contain a common color in their lists. The following papers address a general setting for set-coloring Ramsey numbers: Chung and Liu \cite{ChLi}, Bustamante and Stein \cite{BS}, Lee \cite{L}, He and Mao \cite{HM}, Xu, Shao, Su, and  Li \cite{XSSL}, Conlon, J. Fox, X. He, D. Mubayi, A. Suk and J. Verstra\"ete \cite{CFHMV}, and Arag{\~a}o, Collares, Marciano, Martins, and Morris \cite{ACMMM}.

The paper is structured as follows. 
We prove Propositions \ref{cliques}, \ref{matchings-paths}, and \ref{stars} in Section \ref{sec:cliques}, \ref{sec:matchings-paths}, and \ref{sec:stars}, respectively. 
We consider \bad ~graphs, prove Lemmas \ref{f=n} and \ref{xi=n}, as well as other related results in Section \ref{sec:bad}.
Proposition \ref{f(n,H)} is proved in Section \ref{sec:f(n,H)}. 
We state concluding remarks and open questions in Section \ref{sec:conclusions}. Throughout the paper, we shall write that a graph is {\it bi-colored} if it has at most two colors on its edges, otherwise it is $3$-colored.

\section{Results for $f(H)$ and $\xi(H)$}
\label{sec:f-xi}

\subsection{$H$ as a clique}\label{sec:cliques}

\begin{proof}[Proof of Proposition \ref{cliques}]
First we show that for any $t\geq 5$,  $\xi(K_t)\geq \sqrt{1.5}^t$ by using a classical probabilistic bound, see also the Master's thesis of the first author \cite{A-thesis}. Let $n=a^t$, $a= \sqrt{3/2}\approx 1.224$.
Let $c: E(K_n) \rightarrow \{r,b,g\}$ be a coloring such that each edge gets each of the colors with probability $1/3$ independently and uniformly.
If $S$ is a set of $t$ vertices, the probability that it does not induce any edges of some color is at most $3 (2/3)^{\binom{t}{2}}$.
Then the probability that there is such a bad set $S$ inducing at most two colors is at most 
$$
\binom{n}{t} 3 \left(\frac{2}{3}\right)^{\binom{t}{2}} \leq 
3\left(\frac{ne}{t}\right)^t \left( \frac{2}{3}\right)^{\binom{t}{2}} =
3 a^{t^2} e^t t^{-t}  \left( \frac{3}{2}\right)^{\frac{-t^2}{2} + \frac{t}{2}} =
3 e^t t^{-t} \left(\frac{3}{2}\right)^{ t/2}  <1.$$
Thus the general lower bound on $\xi(K_t)$ and $f(K_t)$  follows. 
The general upper bound on $\xi(K_t)$ was proved by Ortlieb \cite{O-thesis} using a neighborhood-chasing argument, which proceeds by recursively finding monochromatic stars.
The general upper bound on $f(K_t)$  follows from (\ref{xi}) and the best known upper bound on Ramsey numbers $R(K_t)$ by Campos, Griffiths, Morris, and Sahasrabudhe \cite{CGMS}.
See also earlier improvements by Gupta, Ndiaye, Norin, and Wei \cite{GNNW}, Conlon \cite{C2009}, and the original paper by Ramsey \cite{R}.

Now, we treat the small values of $t$.

The edge set of $K_4$ can be decomposed into $3$ pairwise disjoint complete matchings. Color each of them a different color.
In this coloring, every copy of $K_3$ uses $3$ colors. Thus $\xi(K_3)\geq 4$.  To see that $\xi(K_3) <5$, consider  an arbitrary $3$-edge coloring of $K_5$. Since each vertex is incident to four edges,  some two of these edges have the same color and thus belong to a triangle colored with at most two colors.

    By (\ref{xi})  we have that  $f(K_3)\leq R(K_3)=6$.
    Now we present a
    Builder strategy for $n=5$.
    First, expose all four edges incident to a vertex $v$.
    Once two edges of the same color appear, forbid their color on the remaining edges incident to $v$. Then, without loss of generality, $v$ is incident to two red and two blue edges or 
    $v$ is incident to two red, one blue, and one green edge.

    Suppose first  $v$ has two red and two blue incident edges.
    Next, expose the two edges that would complete these monochromatic pairs into triangles, forbidding red and blue on them.
    At this stage, all monochromatic triangles containing $v$ are avoided, allowing Builder to concentrate on the $K_4$ induced by the remaining four vertices.
     At this point, only two independent edges of the $K_4$ have been exposed, using at most two colors.
     If these two edges get the same color, that must be green, forbid green on the remaining edges. 
     If these two edges have different colors, forbid each of these colors on a pair of disjoint edges that have not been exposed yet. 
     Suppose $v$ is incident to two red, one blue, and one green edge. Forbid red on an edge forming a triangle together with two red edges incident to $v$. Proceed similarly to the previous case.

The fact that $\xi(K_4) = 9$ is proved in Theorem 3.6 by Chung and Liu, \cite{CL}. It was reproved several times, see for example Erd\H{o}s and Gy\'arf\'as \cite{EG} (given without proof),  Ortlieb \cite{O-thesis},  Lee \cite{L}, and  Gy\'arf\'as \cite{G}.
We include the lower bound here for completeness. Consider the set of nine vertices $\{0, 1, 2\}^2$. Let $c((x,y), (x',y'))$ be red if $x=x'$, ~ $c((x,y), (x',y'))$ be blue if $y=y'$, and $c((x,y), (x',y'))$ be green otherwise. Let $S$ be an arbitrary set of four vertices. By the pigeonhole principle, there are two vertices in $S$  that coincide in the first coordinate, and there are two vertices in $S$ that coincide in the second coordinate. Thus $S$ induces a red and a blue edge. Since the union of a red and blue color class is  $K_4$-free, $S$ also induced a green edge. 
The upper bound $f(K_4) \leq 18$  follows from (\ref{xi}) and the fact that $R(K_4)=18$, which was published by Greenwood and Gleason \cite{GrGl}.

The statement $19 \leq \xi(K_5) \leq 28$ follows immediately from the work of Ortlieb \cite{O-thesis}, which, together with the Ramsey result $R(K_5)\leq46$ by Angeltveit and McKay \cite{AM}, also gives the bounds for $f(K_5)$.
\end{proof}

\subsection{$H$ as a matching or a path}\label{sec:matchings-paths}

\begin{figure}[h!]
    \centering
    \includegraphics[width=0.5\textwidth]{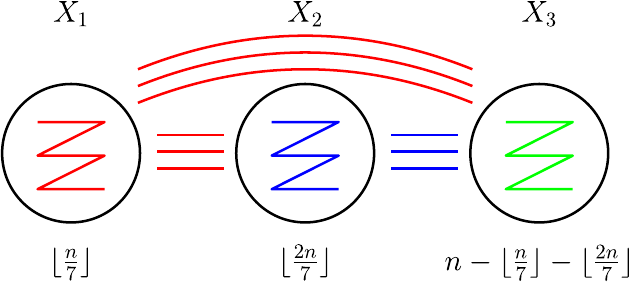}
    \caption{Polychromatic coloring of $K_n$ with $H=M_{\lfloor\frac{3}{7}n \rfloor+1}$}
    \label{fig:xi-mt}
\end{figure}

\begin{proof}[Proof of Proposition \ref{matchings-paths}]
    The lower bound for matchings follows from the result of Goldwasser and Hansen \cite{GH}, who have constructed a polychromatic coloring for any $K_n$ with $n \geq 3$ and $H=M_{\lfloor\frac{3}{7}n \rfloor +1}$.
    Namely, partition the vertex set as $V=X_1\dot{\cup} X_2\dot{\cup} X_3$ with $|X_1|=\lfloor\frac{n}{7}\rfloor$ and $|X_2|=\lfloor\frac{2n}{7}\rfloor$.
    Color each edge with an endpoint in $X_1$ red, color each remaining edge with an endpoint in $X_2$ blue, and color all other edges green. 
    See this construction in Figure \ref{fig:xi-mt}.
    The upper bound for matchings follows from (\ref{xi}) and the result of Cockayne and Lorimer on the Ramsey number of matchings \cite{CL}.

Note that every path $P_t$ contains a matching $M_{\lfloor t/2 \rfloor }$.
Moreover, a polychromatic coloring of $K_n$ for some graph $H$ is also a polychromatic coloring for any graph that contains $H$ as a subgraph.
Hence, the lower bound for paths follows from the first part of the proposition.
The upper bound comes from $(\ref{xi})$ and the Ramsey result of Gerencsér and Gyárfás \cite{GG}, which states that $R(P_t) = \lfloor\tfrac{3t}{2} \rfloor -1$.
\end{proof}

\subsection{$H$ as a star $K_{1,t}$}\label{sec:stars}

A $3$-edge coloring of a graph is {\it balanced} if for every vertex $v$ and any two colors, the number of edges colored with the first color incident to $v$ differs from the number of edges colored with the second color incident to $v$ by at most one.
We shall need an auxiliary lemma on balanced colorings.
\begin{lemma}\label{balanced}
For any $n\geq 3$ there is a balanced $3$-edge coloring of $K_n$.
\end{lemma}

\begin{proof}
    We shall construct a red, blue, green edge coloring depending on whether $n$ is even, $n$ is odd and $n-1$ is divisible by $3$, and two remaining cases.  
    
  \textbf{Case 1:}  $n$ is even. The edge set of $K_n$ can be decomposed into $n-1$  spanning matchings.
  Split the set of these matchings into three almost equal parts, and color the edges of the matchings from distinct parts in distinct colors, red, blue, and green.  
  
\textbf{Case 2:} $n$ is odd and $n-1$ is divisible by $3$.
There is an edge-decomposition of $K_n$ into $(n-1)/2$  Hamiltonian cycles, for example, Walecki's decomposition \cite{A}. Split the set of these Hamiltonian cycles into three equal parts, and color the edges of the cycles from distinct parts in distinct colors red, blue, and green. 

\textbf{Case 3:} $n-1= 6k+2$ for an integer $k$.
Fix an edge-decomposition of $K_n$ into $(n-1)/2$ Hamiltonian cycles just as in the second case. 
Split $3k$ out of these $3k+1$ Hamiltonian cycles into three equal parts, and color the edges from distinct parts in distinct colors.
Then consider the last cycle.
Color one edge of this cycle red, and the remaining edges alternately blue and green.  
With this coloring, the distribution of colors at each vertex in $K_n$ is balanced.

\textbf{Case 4:}  $n-1= 6k+4$ for an integer $k$.
Fix a vertex $v$ and consider a balanced coloring of $K'=K_n-\{v\}$, as in the first case.  In particular, each vertex of $K'$ is incident to $2k+1$ edges of each of the three colors.
Color $2k+1$, $2k+1$, and $2k+2$ edges incident to $v$ in red, green, and blue, respectively.  

In each of these cases, the coloring is balanced.
\end{proof}

\begin{proof}[Proof of Proposition \ref{stars}]
First we shall bound $\xi=\xi(K_{1,t})$. 
He and Mao \cite{HM} proved set-coloring Ramsey results for stars that imply similar bounds, namely $\lfloor \frac{3t-3}{2}\rfloor \leq \xi(K_{1,t}) \leq \lceil \frac{3t+2}{2} \rceil  $.
However, we also include our proof for completeness.

Let $n= \lfloor\frac{3t-1}{2} \rfloor$. Note that for $t\geq 3$, $n \geq 4$. To prove the upper bound on $\xi$, we observe that every $3$-edge coloring of $K_{n+1}$
    necessarily contains a bi-colored copy of $K_{1,\lceil  2n/3  \rceil }$ by the pigeonhole principle applied to the edges incident to an arbitrary vertex.
    Since $\lceil  2n/3  \rceil  = \lceil  \frac{2}{3}  \lfloor\frac{3t-1}{2}  \rfloor  \rceil  = t$, we have a bi-colored copy of $K_{1,t}$.  For the lower bound on $\xi$, we construct a  balanced $3$-edge coloring of
    $K_n$  that is given by Lemma \ref{balanced}.  
    In this coloring  any bi-colored star has at most $\lceil \frac{2}{3} (n-1) \rceil =\lceil \frac{2}{3} ( \lfloor\frac{3t-1}{2} \rfloor-1) \rceil $ edges, that is strictly less than $t$ edges. 
That means that any star on $t$ edges uses all three colors.

  Now, we shall bound $f=f(K_{1,t})$.
  The upper bound on $f$ follows from (\ref{xi}) using the Ramsey number $R(K_{1,t}) = 2t$,  see Harary \cite{H}.
    For the lower bound, we show that Builder can expose all edges of $K_{2t-2}$ so that no monochromatic $K_{1,t}$ appears.
    We prove this by induction on $t$.

For $t=2$, we have that  $f(K_{1,2})=3$. To see the lower bound,  observe that  $K_2$ does not contain $K_{1,2}$ as a subgraph. On the other hand, Painter could win on $K_3$ as follows. After two edges of $K_3$ have been exposed and say got colors red and blue, Builder can forbid only one color, so Painter can use either red or blue to complete a monochromatic $K_{1,2}$.

Assume Builder has a strategy to expose all edges of $K_{2(t-1)-2}$ while avoiding a monochromatic star $K_{1,t-1}$.
Assume that  Builder has applied this strategy on a set $X=[2t-4]$  of vertices. 
Next, Builder exposes edges between $X$ and  new vertices $\{a,b\}$ in $t-2$ steps such that in $j^{th}$ step he exposes the edges between $\{a,b\}$ and $\{2j-1, 2j\}$.
Let for a vertex $v$ and a set of vertices $Y$, $\Delta(v, Y)$ be the size of the largest monochromatic star with center $v$ and leaf-set $Y$.
Builder plays such that \\
1. $\Delta(a, [2i])\leq i+1$, $\Delta(b, [2i])\leq i+1$, and \\
2. each vertex in $X$ is joined by edges of different colors to $a$ and $b$.

For $i=1$, the color degree condition is trivially satisfied, and the second condition can also easily be satisfied by forbidding the color of $a1$ on $b1$ and the color of $a2$ on $b2$.\\
Assume that Builder ensures the validity of conditions 1. and 2. in step $j$. Consider now step $j+1$ when the edges between $\{a,b\}$ and $\{2j+1, 2j+2\}$ are being exposed.
If $\Delta(a, [2j])\leq j$, then Builder only makes sure that the colors on  $a(2j+1)$ and $b(2j+1)$  are distinct and colors  on $a(2j+2)$ and $b(2j+2)$ are distinct. 
By exposing the two edges incident to $b$ first, he could make sure that these edges have different colors as well.
We have that $\Delta(a, [2j+2])\leq j+2 = (j+1)+1$ and $\Delta(b, [2j+2]) \leq  (j+1)+1$.

If $\Delta(a, [2j])= \Delta(b, [2j]) = j+1$, we see that the monochromatic stars of sizes $j+1$ with centers $a$ and $b$ and leaves in $[2j]$ have a common leaf in $X$  and thus have distinct colors, say red and blue, respectively.
Note also that all other monochromatic stars between $a$ and $[2j]$ and between $b$ and $[2j]$ have size at most $j-1$.
Builder proceeds as follows:  exposes $a(2j+1)$ and forbids red, exposes $b(2j+1)$ and forbids the color of $a(2j+1)$, 
    exposes $b(2j+2)$ and forbids  blue, exposes $a(2j+1)$ and forbids the color of $b(2j+2)$. See Figure \ref{fig:g-k1t}.
Then each of $\Delta(a, [2j])$ and $\Delta(b, [2j])$ increases by at most one, so condition 1. is satisfied after step $j+1$.

\begin{figure}[h!]
    \centering
    \includegraphics[width=0.4\textwidth]{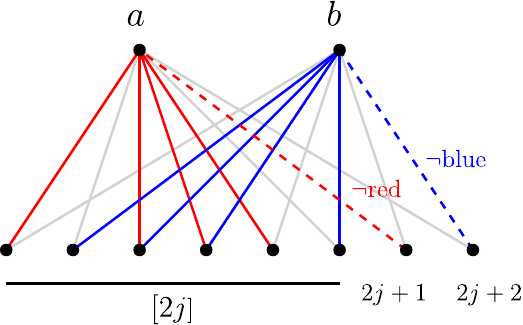}
    \caption{Builder strategy for $H=K_{1,t}$ and $j=3$}
    \label{fig:g-k1t}
\end{figure}

At this point, Builder has exposed all edges of  $K_{2t-2}$ with vertex set $X\cup \{a,b\}$ except for $ab$,  without creating a monochromatic $K_{1,t-1}$ with center in $\{a,b\}$ and without monochromatic 
$K_{1, t}$ with center in $X$.  Now Builder exposes $ab$ and forbids a color on it arbitrarily.  The largest monochromatic star in a resulting coloring has at most $t-1$ edges.
\end{proof}

\section{\bad ~graphs and graphs with $f(H)>|V(H)|$}\label{sec:bad}

 Recall, that  a graph $H$ on $n$ vertices is called \emph{\bad}, if every $3$-coloring of edges of $K_n$ contains a bi-colored copy of $H$, i.e., if $\xi(H)=n-1$. 
 In other words, there is no $3$-polychromatic coloring for $H$ of any clique on at least $|V(H)|$ vertices.  
 If a graph $H$ is not \bad, we call it {\it \good}. Note that if a  subgraph of $H$ is \good, then $H$ is also \good. Thus we concentrate on \good{} and \bad{} trees in this section.

We shall show that some \bad{} graphs exist. For example we prove that some small brooms and double stars are \bad{} by an exhaustive case analysis.
 Note that $P_t$ is non-\bad~ for $t\geq 7$ as follows from Proposition \ref{matchings-paths}.

Next we shall prove Lemma \ref{xi=n} claiming that any 
 $H\in \{P_4, P_5, P_6, B_{3,3}, B_{3,4}, S_{2,2}\}$ is \bad.

\begin{proof}[Proof of Lemma \ref{xi=n}]
We shall consider three cases for paths, for brooms, and for the double star, respectively.
For each of these graphs we consider an arbitrary $3$-coloring of the edges of a complete graph on the same number of vertices and show that there is a bi-colored copy of the considered graph. We denote the vertices of the colored complete graph by $v_1, v_2, \ldots$.

\textbf{Case 1.} Paths.

\textbf{Case 1.1}  $P_4$.
 Observe that any $3$-coloring of the edges of $K_4$ contains two edges of the same color. These edges form a subgraph of a path on three edges that is bi-colored.

\textbf{Case 1.2}  $P_5$.
Any $3$-coloring of the edges of $K_5$ contains four edges of the same color. Either three of these four edges are on a path, or these four edges form a star. In both cases there is a bi-colored $P_5$.

\textbf{Case 1.3}  $P_6$.
   Fix a $3$-coloring of the edges of $K_6$ and suppose for contradiction that every copy of $P_6$ uses all three colors.
    Let $M$ be an arbitrary set of three independent edges and decompose $E(K_6)$ into $M$ and another four disjoint perfect matchings.
    The union of any two such matchings forms a $C_6$, which must then use all three colors twice. 
    Hence, all five matchings must contain three edges of three different colors.
    Moreover, any three independent edges of $K_6$ must also use all three colors.
    So no two independent edges share a color, a contradiction.

\vskip 0.6cm

\textbf{Case 2.} Brooms.

\textbf{Case 2.1.}  $B_{3,3}$.
Consider a $3$-coloring of $K_6$ and assume without loss of generality  that each edge $v_1v_i$, $i\in \{2,3,4,5\}$ is either blue or green. 
    If  $v_6v_i$ is blue or green for some $i\in \{2,3,4,5\}$, then the four edges $v_1v_2$, $v_1v_3$, $v_1v_4$, and $v_1v_5$ together with $v_6v_i$  form a bi-colored $B_{3,3}$. Otherwise all four edges $v_6v_2$, $v_6v_3$, $v_6v_4$,  and $v_6v_5$ are red.
    But then, these four red edges together with the edge $v_1v_2$ yield another bi-colored copy of $B_{3,3}$.

\textbf{Case 2.2.}  $B_{3,4}$.
Consider a $3$-coloring of the edges of $K_7$. We shall show that there is a bi-colored $B_{3,4}$.

 \begin{description}
\item{Case 2.2.1.}  There is a monochromatic spanning star. Let its center be $v$. For any $u\neq  v$,  there are two vertices $w, w'$ not equal to $v$ that are joined to $u$ by the edges of the same color. Then $w', w, u, v$ form a stick of a bi-colored broom $B_{3,4}$.

\item{Case 2.2.2.} There is a bi-colored spanning star. Lets its center be  $v$ and its colors be red and green.
If there is a red/green $P_3$ induced by $N(v)$, we are done. Otherwise red/green graph induced by $N(v)$ is a matching, so there is a $B_{3,4}$ in blue/green or in blue/red.

\item{}{Case 2.2.3.} There is a vertex $v$ incident to four edges of the same color, say blue.
Let  $vw$ be green, $vw'$ be red. Let $X=N(v)-\{w, w'\}$. If there is a blue edge between $\{w, w'\}$ and $X$, we have a bi-colored $B_{3,4}$.
 Thus we have a red-green $K_{2,5}$ with parts $\{w, w'\}$ and $\{v\} \cup X$.   If $X$ induces a red or green edge, we have a red/green $B_{3,4}$. So, $X$ induces all blue edges and we can assume that there is no blue edge incident to $w$. 
This brings us to Case 2. due to $w$.

\item{Case 2.2.4.}  There is a vertex $v$ incident to three edges of one color, two edges of another color, and one edge of the third color, say red.  Let $vw$ be red.  The vertex $w$ must have a blue and a green incident edge, which together serve as the stick of a blue-green $B_{3,4}$ with center $v$.

\item{Case 2.2.5.} Each vertex is incident to two edges of each color. 
In this case each color class is a $2$-factor, that is either a $C_7$ or a vertex-disjoint union of $C_3$ and $C_4$. 
Let $v$ be a vertex of red $C_4$ or a red $C_7$. Then  $v$ is the center of a blue/green star on $4$ edges, and there is a blue/green edge vertex-disjoint from this star, so we have a blue/green $B_{4,3}$ with center $v$.
\end{description}

\vskip 0.6cm

\textbf{Case 3.} Double star $S_{2,2}$.\\
Consider a $3$-edge coloring of $K_6$.
    We proceed by case distinction, see the different cases in Figure \ref{fig:bad-s22}.
    \begin{figure}[h!]
    \centering
    \includegraphics[width=0.6\textwidth]{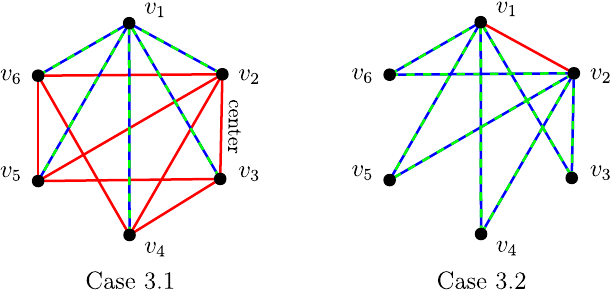}
    \caption{$S_{2,2}$ is \bad.}
    \label{fig:bad-s22}
\end{figure}

\begin{description}
    \item{Case 3.1.}
         There is a bi-colored spanning star. Assume $v_1$ has no red incident edges.
        If the non-red subgraph induced by the set $\{v_2,v_3,v_4,v_5,v_6\}$ has a vertex of degree at least two, there is a bi-colored $S_{2,2}$.  Otherwise the set $\{v_2,v_3,v_4,v_5,v_6\}$ induces a red copy of $K_5-2K_2$.
        As illustrated in Figure \ref{fig:bad-s22}, this configuration necessarily contains a bi-colored copy of $S_{2,2}$ using $4$ red edges and $1$ non-red edge.

    \item{Case 3.2.} There are no bi-colored spanning stars.
    Assume that $v_1v_2$ is red and all other edges incident to $v_1$ are not red. 
    If $v_2$ has only one red incident edge, then from the Figure 3 we see that the vertex set $\{v_3, v_4, v_5, v_6\}$ induces only red edges. Let without loss of generality $v_1v_3$ and $v_1v_4$ be of the same color, say blue. Then we have a bi-colored $S_{2,2}$ with center edge $v_1v_4$.

    Thus we can assume that every color class has no isolated edges.
    We  also see that each color class is a spanning subgraph that is connected, otherwise the other two color classes together would induce an $S_{2,2}$. 
    Thus each color class is a spanning tree since there are 15 edges in $K_6$. 
    If $v_2$ is incident to at least three non-red edges, their endpoints must induce a red triangle, a contradiction.
    Thus $v_2$ is incident to exactly three red edges. Therefore, some neighbor of $v_2$, say $v_5$, in the red subgraph has degree at least two. This implies that there is a bi-colored $S_{2,2}$ with center edge $v_2v_5$. \end{description} 
    
    This concludes the proof of Lemma \ref{xi=n}\end{proof}

It follows from Propositions \ref{matchings-paths} and \ref{stars} that matchings, long paths, and stars are not \bad. Here, we provide wider classes of not \bad~trees, i.e. those trees $T$ for which there is a $3$-polychromatic coloring of $E(K_{V(T)})$ with each copy of $T$ having all three colors. Let $\Delta(H)$ be the maximum degree of $H$ and $\tau(H)$ be the smallest size of a vertex set $S$ such that any vertex of $H$ is either in $S$ or adjacent to a vertex in $S$, i.e., the size of a smallest dominating set.  For a tree $T$, let $\ell(T)$ be the largest number of leaves incident to a single vertex. For a vertex set $X$ in a graph, we define the neighborhood of $X$, $N(X)$, as the set of all vertices in $V\setminus X$ that are adjacent to some vertex in $X$. The logarithm $\ln$ below is the natural logarithm. \\

\begin{lemma}\label{spanning-trees}
A tree $T$ on $n$ vertices is \good{} if it satisfies one of the following\\
1.  $\tau(T) \geq 3 +\ell(T)$.\\
2. $\Delta(T) > \lceil \tfrac{2}{3}(n-1) \rceil$ and $n \geq 4$.\\
3.  For $n$ sufficiently large,  $a, b \in \mathbb{N}_0$, such that $\binom{n}{a}\binom{n-a}{b}(1 - 3^{-3a})^{\lfloor (n-a-b)/3 \rfloor} <1$,  there is a set $X$  of $a$ vertices  of $T$ with $|N(X)|= n-a-b$. 
In particular, any $1\leq a< \ln \ln n$ and $0\leq b < n \ln^{-5} n$ can be taken.\\
4. $n\geq 122$ and $\tau(T)=2$.
\end{lemma}

\begin{proof}
1. Let's define a coloring of $K_n$  for a positive integer $a\leq n-3$ as follows.
Let  $V_a$ be a subset of vertices of size $a$ and $v \in V(K_n)\setminus V_a$.
Color all edges incident to $v$ red, other edges that are incident to $V_a$ blue, and the remaining edges green.
We see that for any tree $T$, $T$ has red on some of its edges since $T$ is connected. If $a>\ell(T)$, each copy of $T$ has a blue edge.
Finally, if $\tau(T)>1+a$, there is a green edge in any copy of $T$. Note that if $\tau \geq 3+\ell$, such an integer $a$ exists.
 Hence every copy of $P_n$ is $3$-colored.\\

2. Suppose $\Delta(H) > \lceil \tfrac{2}{3}(n -1)\rceil$ with $n \geq 4$.
Consider a balanced  $3$-edge coloring of $K_n$, which exists by Lemma \ref{balanced}.   
By pigeonhole principle any bi-colored star in $K_n$  has at most  $ \lceil \tfrac{2}{3}(n-1) \rceil$ edges.  Since $H$ contains a star on more than  $ \lceil \tfrac{2}{3}(n-1) \rceil$ edges,   all copies of $H$ are $3$-colored.

3.  Figure \ref{fig:bad-prob} illustrates the proof of this part. 

\begin{figure}[h!]
    \centering
    \includegraphics[width=0.4\textwidth]{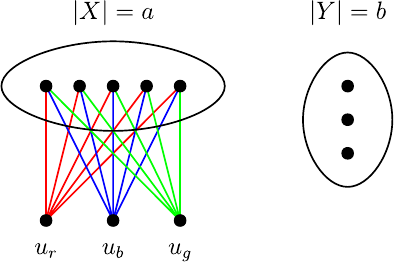}
    \caption{Probabilistic proof for \good{} graphs.}
    \label{fig:bad-prob}
\end{figure}

Consider a random $3$-edge coloring of $G=K_n=(V, E)$ with colors assigned to edges independently and uniformly, that is, the probability that an edge gets a specific color is $1/3$.
We say that a copy of $T$ is {\it good } if it is $3$-colored.
For any $a$-element set of vertices $X'$ in $G$ and a tuple $(u_r, u_b, u_g)$ of vertices of $G$ not in $X'$, we say that $(X', u_r, u_b, u_g)$ is {\it good} if all edges between $X'$ and $u_r$ are red,
all edges between $X'$ and $u_b$ are blue, and  all edges between $X'$ and $u_g$ are green.  Otherwise, we say that $(X', u_r, u_b, u_g)$ is {\it bad}.
Note that if $(X', u_r, u_b, u_g)$ is good, then any copy of $T'$ of $T$ with $X(T)$ corresponding to $X'$ and $u_r, u_b, u_g \in N_{T'}(X')$  is good.
Let $p':= Prob((X', u_r, u_b, u_g) \mbox{ is  bad})$.  Then $p' =1 - 3^{-3a}$.
We have then that the probability that some copy of $T$ is bad is at most $\binom{n}{a} \binom{n-a}{b}  \cdot p$, where
$\binom{n}{a} $ is the number of ways to choose a copy $X'$ of $T(X)$, $\binom{n-a}{b} $ is the number of ways to choose the copy $Y$ of $V(T)\setminus (X(T) \cup N(X(T)))$, and $p$ is the probability that for  
any ordered triple $(u_r, u_b, u_g)$ of vertices from $V\setminus (X'\cup Y)$, $(X', u_r, u_b, u_g)$ is bad.
We can bound $p$ as follows
$p \leq p'^{\lfloor (n-a-b)/3 \rfloor} \leq (1 - 3^{-3a})^{\lfloor (n-a-b)/3 \rfloor}$, by splitting the set $V-(X\cup Y)$ into pairwise disjoint triples corresponding to $u_r, u_b, u_r$.
Then the probability that some copy of $T$ is bad is at most $\binom{n}{a} \binom{n-a}{b}  \cdot p <1$ for sufficiently small $a$ and $b$. 
The resulting inequality $\binom{n}{a}\binom{n-a}{b}((1 - 3^{-3a})^{\lfloor (n-a-b)/3 \rfloor}) <1$ holds for fixed $a$ and $b$ and growing $n$ or, for example, for 
$1\leq a< \ln \ln n$ and $0\leq b < n \ln^{-5} n$.
We include the respective inequalities in the Appendix.

4.  We shall consider a random coloring of $K_n$ as in 3.  If $\tau(T)=2$, then $T$ satisfies one of the following cases.

Case 1. $T$ is a double star.\\
Let $p=p_{xy}$ be the probability of a bad event that some copy of $T$ with center vertices $x$ and $y$ is not rainbow.  We shall condition on the color of $xy$, say that the color of $xy$ is red. Then $p$ is upper bounded by the probability that each vertex  $z$ outside of $x,y$ does not send both blue edges to $x$ and $y$ plus the probability that each vertex $z$ outside of $x, y$ does not send both green edges to $x$ and $y$.
That is $p\leq 2(8/9)^{n-2}$.
Now, taking the union bound over all $x$ and $y$, we have that the probability that there is a non-rainbow copy of $T$ is at most $\binom{n}{2} 2(8/9)^{n-2}$.
This expression is strictly less than one for any $n\geq 76$.

Case 2. $T$ is a union of two stars that share one leaf. Let the stars' centers be $x'$ and $y'$, and their common leaf be $z'$. \\
Let $p=p_{xyz}$ be the probability of a bad event that some copy of $T$ with vertices $x,y,$ and $z$ corresponding to $x', y',$ and $z'$ is missing some of the colors on its edges.
Conditioning on the possible colors of the edges $zx$ and $zy$, we have that 
$p_{xyz}\leq 
(6\cdot 1/9  + 3 \cdot 1/9 \cdot 2) (8/9)^{n-3}$. Here $(8/9)^{n-3}$
is the probability that each vertex outside of $x, y$, and $z$ does not send two edges of the same given color to both $x$ and $y$. 
Taking the union bound over all $x,y$ and $z$, we have that the probability that there is a non-rainbow copy of $T$ is at most 
$  (n \binom{n-1}{2})(4/3)(8/9)^{n-3}$.
This expression is strictly less than one for any $n\geq 122$.
\end{proof}

\begin{lemma}
For any connected graph $H$ on $n \ge 122$ vertices, Builder wins on $n$ vertices, i.e., $f(H)>n$. Moreover, for any graph $H$ on $n\geq 6$ vertices containing a spanning tree $T$ with $\tau(T)\geq 3$, Builder wins on $n$ vertices.
\end{lemma}

\begin{proof}
   Let $H$ be any graph on $n$ vertices, and let $T$ be a spanning tree of $H$.
    If Builder can ensure a win for $T$, then Builder can do so for $H$ as well: by exposing all edges of $K_n$, Builder can guarantee that no monochromatic copy of $T$ appears, and consequently no monochromatic copy of $H$ can appear either.

    Assume that $H$ has a spanning tree $T$ with $\tau(T)=2$. The result holds by Lemma \ref{spanning-trees} part 4.

    Next assume that for some spanning tree $T$ of $H$, $\tau(T)\geq 3$.
    Builder shall use the following strategy. Let $x$ and $y$ be two vertices of $K_n$.  First, expose the edge $xy$, and w.l.o.g., assume that Painter colors it red. 
    Next, expose all remaining edges incident to $x$ while forbidding blue, and all remaining edges incident to $y$ while forbidding green. 
    Finally, expose all remaining edges of $K_n$, this time forbidding red.
    Under this construction, no green or blue copy of $T$ can appear: the vertex $x$ has no incident blue edges, and $y$ has no incident green edges. 
    Consequently, any monochromatic tree on $n$ vertices in the resulting coloring must be red. 
    Moreover, every red edge is incident to either $x$ or $y$. Since $\tau(T)\geq 3$,  there is no monochromatic copy of  $T$ and thus no monochromatic copy of $H$.    
\end{proof}

\section{Results about $f(n,H)$ and $\xi(n,H)$}
\label{sec:f(n,H)}

Recall that the Tur\'an graph $T(n,k)$  is a balanced complete $k$-partite graph on $n$ vertices, one can also look at it as a blowup of $K_k$ with almost equal parts.

\begin{proof}[Proof of Proposition \ref{f(n,H)}]
From (\ref{eq:lb}) we have $f(n, H)\leq \ex_2(n, H).$ It is known, see for example  \cite{HST}, \cite{AGL}, \cite{G-thesis}, that  $$\ex_2(n, H) \leq \left( 1- \frac{1}{R_\chi(H) -1}\right)\binom{n}{2} (1+o(1))$$ 
and $R_\chi(K_t) = R(K_t)$. This gives the upper bounds for non-bipartite graphs.

For the lower bound on $f(n, H)$ for non-bipartite $H$, consider a $3$-coloring $c$ of the edges of $K_k$, where  $k:=\xi(K_{\chi(H)})$ such that each clique on $\chi(H)$ vertices has all three colors. For sufficiently large $n$ consider $T(n, k)$ with an edge coloring inherited from $c$, i.e., such that all edges between two parts of this Tur\'an graph have the same color as the corresponding edge of $K_k$. Consider a subgraph $H'$ of $T(n,k)$ isomorphic to $H$. There are some vertices of $H'$ in at least $\chi(H)$ parts of $T(n,k)$.  The respective clique on $\chi(H)$ vertices in $K_k$ has all three colors. There are all three colors between the corresponding parts in $T(n,k)$ and thus in $H'$.

If $H$ is bipartite, it is known (and easily follows from a random packing of two extremal graphs) that $\ex_2(n, H) \leq 2 \ex(n,H)$, see for example \cite{AGL} or \cite{LMZ}.
The lower bound $f(n,H)\geq \ex(n,H)$ holds for any graph $H$ since Builder can expose edges of an $H$-free graph with $\ex(n, H)$ edges.
\end{proof}

\section{Concluding remarks}\label{sec:conclusions}
We considered the online Ramsey turnaround game for a graph $H$, where Builder exposes edges of an $n$-vertex graph one at a time and forbids one of the colors red, blue, or green. Painter colors that edge with one of the two colors allowed by Builder. The game ends when Painter creates a monochromatic copy of $H$.  We showed that the smallest number $f(n,H)$ of edges needed by Painter to win the game is closely related to the two-color Tur\'an numbers and polychromatic colorings of graphs.

Mirbach defined this game in greater generality with a given set of $q$ colors and a given number $f$ of colors that Builder can forbid, for $f<q$.  
For many sets of parameters $(q,f)$, the problem could be reduced to a setting with a smaller $q$ or a different $f$, see Mirbach \cite{M-thesis} and Almási \cite{A-thesis}. 
The first non-trivial setting for the parameters is $(q,f)=(3,1)$, exactly as in our definition in this paper. 

We proved a number of results on $3$-polychromatic colorings for $H$, i.e., those edge colorings of a ground graph on at least $|V(H)|$ vertices in $3$ colors, where each copy of $H$ contains all three colors. This led to the definition of $\xi$-primitive or polychromatic-primitive graphs $H$ that are the graphs such that no clique has a $3$-polychromatic coloring for  $H$. We found three small trees that are polychromatic-primitive. The following remains:

{\bf Question 1.} Are there arbitrarily large \bad{}  connected graphs? 

Some natural candidates for being \bad{} include complete binary trees, brooms, comb graphs, or spiders.  

Note that the answer to this question is "no" if we replace $3$ colors with $2$  in the definition of \bad{} graphs. This is because for every connected graph $H$ on $n \geq 4$ vertices, there exists a $2$-edge coloring of $K_n$ in which every copy of $H$ uses both colors.
Indeed, let $T$ be a spanning tree of $H$.
If  $T$ is a star, color the edges of a cycle $C_n$ red and all remaining edges of $K_n$ blue.  
Then every vertex has incident edges in both colors, so every spanning star in $K_n$  uses both colors. If $T$ is not a star,  color all edges incident to some vertex $v$  of $K_n$ red and all other edges of $K_n$ blue. Any copy of $T$ contains $v$, and hence a red edge.  Since $T$ is not a star, every copy of $T$ also contains an edge not adjacent to $v$, which is blue.

We showed that the function $\xi(H)$ provides a lower bound on $f(H)$, but these functions could differ, as for example in the case of $H$ being a star. 

We also proved that Painter could not win on $n$ vertices for any connected graph on $n$ vertices and $n$ being sufficiently large, $n\geq 122$.  On the other hand, one can  see that Painter wins on $|V(H)|$ vertices for $H\in \{P_2, P_3\}$. 

{\bf Question 2.} Is there a connected graph $H$ on $n$ vertices, $4\leq n\leq 121$, such that Painter wins on $n$ vertices, i.e., such that $f(H)=n$?

Each of the four considered functions $\xi(H)$, $\xi(n, H)$, $f(H),$ and $f(n,H)$ are of independent interest, in particular for cliques. Our bounds imply that 
$\frac{3}{4} \binom{n}{2} (1+o(1)) \leq f(n,K_3)\leq \frac{4}{5} \binom{n}{2} (1+o(1)).$

{\bf Question 3.} What is the value of $f(n, K_3)$?

In Lemma \ref{spanning-trees} we showed that with positive probability a random coloring of edges of $K_n$ is 3-polychromatic for spanning trees with sufficiently high maximum degree.  We remark that the random coloring approach has its limitations and works indeed only for such types of spanning trees. Indeed, a random 3-coloring of $E(K_n)$, where each edge is assigned one of the three colors (red, blue, or green) with probability $1/3$ independently has a property (with high probability) that each vertex is incident to at least $(2/3 - \epsilon)n$ edges that are either red or blue.  A result by Koml\'os, S\'ark\"ozy, and Szemer\'edi \cite{KSS} claims that an $n$-vertex graph with minimum degree $(1/2+\epsilon)n$ has any tree with maximum degree at most $n/\ln n$ as a spanning tree. 
A result of Montgomery \cite{M} claims that for any $\Delta>0$ there is a constant $C=C(\Delta)$  such that for any sequence of trees $T_n$, $n\geq 1$, where $T_n$ has $n$ vertices and maximum degree at most $\Delta$, the probability $P(T_n\subseteq G(n,C\log n/n))$ tends to 1 as $n$ tends to infinity.
That is, in particular,  we are very likely to have monochromatic copies of such trees in a randomly colored $K_n$.

We remark that the analogue of the function $\xi(K_t)$ was addressed for $3$-uniform hypergraph by Mulrenin, Pohoata,  and Zakharov, \cite{MPZ}, where they conjectured the growth rate for the respective function to be at most $2^{O(t)}$.

{\bf Acknowledgements}  The authors thank Andr\'as Gy\'arf\'as for inspiring discussions and Felix Joos for useful comments on random graphs.

The first author's work is part of a project supported by the Doctoral Excellence Fellowship Programme (DCEP) and funded by the National Research Development and Innovation Fund of the Ministry of Culture and Innovation and the Budapest University of Technology and Economics.


\section*{Appendix}
\label{sec:app}
We provide additional calculations for the proof of Lemma \ref{spanning-trees}.
Namely, we show that for large $n$ any $1 \leq a < \ln \ln n$ and $0 \leq b <n \ln^{-5}n$ the following holds: 
$$\binom{n}{a}\binom{n-a}{b}(1 - 3^{-3a})^{\lfloor (n-a-b)/3 \rfloor} <1.$$

Observe that $a+b \leq 2 n \ln^{-5}n$ and $\lfloor \frac{n-a-b}{3}\rfloor \geq \frac{n}{4}$. 
Thus 
\begin{eqnarray*}
\binom{n}{a}\binom{n-a}{b}(1 - 3^{-3a})^{\lfloor (n-a-b)/3 \rfloor} 
&\leq &
n^{a+b} \cdot (1 - 3^{-3a})^{n/4} \\
&\leq&n^{a+b} \cdot e^{-3^{-3a} \cdot n/4}\\
&= & \exp\left((a+b) \ln n - 3^{-3a} \tfrac{n}{4} \right)\\
&\leq & \exp\left( \frac{n}{4}( 8 \ln^{-4}n - 3^{-3 \ln \ln n} )\right)\\
&\leq & \exp \left(\frac{n}{4}( 8 \ln^{-4}n  - e^{-3\ln 3\ln\ln n })\right)\\
&\leq & \exp\left( \frac{n}{4} ( 8 \ln^{-4}n - e^{\ln \ln^{-3\ln 3}n})\right)\\
&=& \exp\left( \frac{n}{4}( 8 \ln^{-4}n - \ln^{-3\ln 3}n) \right)\\
&<& 1,
\end{eqnarray*}
since $4>3 \ln 3 \approx 3.29.$
\end{document}